\documentclass[preprint,12pt]{elsarticle}
\usepackage{eurosym}
\usepackage{amssymb}
\usepackage{amsmath}
\usepackage{amsthm}
\usepackage{graphicx,comment}
\usepackage{wrapfig}
\usepackage{color}
\usepackage{enumitem}

\newcommand{\beq}{\begin{eqnarray*}}
\newcommand{\eeq}{\end{eqnarray*}}
\newtheorem{theorem}{Theorem}[section]
\newtheorem{cor}{Corollary}[section]
\newtheorem{lemma}[theorem]{Lemma}

\numberwithin{equation}{section}

\begin{document}
\begin{frontmatter}

\title{Quenching estimates for a non-Newtonian filtration equation with singular
boundary conditions}
\author{Matthew A. Beauregard$^{a}$ and Burhan Selcuk$^{b}$}
\address{$^a$Department of Mathematics and Statistics, Stephen F. Austin State University, Nacogdoches, TX, 75962, USA
\\
$^b$Department of Computer Engineering, Karabuk University, Bal\i
klarkayas\i \ Mevkii, 78050, Turkey
}
\begin{abstract}
In this paper, the quenching behavior of the non-Newtonian filtration equation $(\phi (u))_{t}=(\left \vert u_{x}\right \vert ^{r-2}u_{x})_{x}$ with singular boundary conditions, $u_{x}\left( 0,t\right) =u^{-p}(0,t)$, $u_{x}\left( a,t\right) =(1-u(a,t))^{-q}$ is investigated. Various conditions on the initial condition are shown to guarantee quenching at either the left or right boundary. Theoretical quenching rates and lower bounds to the quenching time are determined when $\phi(u)=u$ and $r=2$. Numerical experiments are provided to illustrate and provide additional validation of the theoretical estimates to the quenching rates and times.
\end{abstract}



\end{frontmatter}

\section{Introduction}

\noindent \noindent In this paper, we study the quenching behavior of the
following nonlinear heat equation with singular boundary conditions:
\begin{equation}
\left \{
\begin{array}{l}
(\phi (u))_{t}=(\left \vert u_{x}\right \vert ^{r-2}u_{x})_{x},\text{\ }%
0<x<a,\ 0<t<T, \\
u_{x}\left( 0,t\right) =u^{-p}(0,t),\ u_{x}\left( a,t\right)
=(1-u(a,t))^{-q},\ 0<t<T, \\
u\left( x,0\right) =u_{0}\left( x\right) ,\ 0\leq x\leq a,%
\end{array}%
\right.  \label{NN_EqSBC}
\end{equation}
where $\phi (s)\ $is an appropriately smooth and strictly monotone
increasing function with $\phi (0)=0,\phi (1)=1\ $and$\  \phi ^{\prime
}(s)\leq 0$. $p$, $q\ $are positive constants, $r\geq 2$ and $T\leq \infty$ and the initial function $u_{0}(x)$ is a non-negative smooth function satisfying the compatibility conditions:
\begin{equation*}
u_{0}^{\prime }\left( 0\right) =u_{0}^{-p}(0),\ u_{0}^{\prime }\left(
a\right) =(1-u_{0}(a))^{-q}\text{.}
\end{equation*}
In the case, $\phi (u)=u^{1/m}\ (0<m<1)$, Eq.~(\ref{NN_EqSBC}) is known as
the classical non-Newtonian filtration equation that attempts to model non-stationary fluid flow through a porous medium where the tangential stress of the fluid's displacement velocity, $u$, has a power dependence under thermodynamic expansion and compression as a result of heat transfer \cite{Li_2012, Mi_1, Wang}. The singular boundary flux terms represent a nonlinear radiation law at the boundary and is common to polytropic filtration equations \cite{Li_2008,Li_2012,Mi_1,Wang}. This mathematical model may exhibit finite-time quenching, defined as a time $T=T(u_{0})<\infty$ such that
\begin{equation*}
\underset{t\rightarrow T^{-}}{\lim }\min \{u(x,t):0\leq x\leq a\}
\rightarrow 0\  \text{or }\underset{t\rightarrow T^{-}}{\lim }\max
\{u(x,t):0\leq x\leq a\} \rightarrow 1\text{.}
\end{equation*}
In the following, the quenching time of Eq.~(\ref{NN_EqSBC}) is denoted as $%
T$.

As is well known, when $\phi(u)=u\ $and $r=2$, the equations reduce to the heat equation. In \cite{Selcuk_2015} Selcuk and Ozalp considered the problem:
\begin{equation}
\left \{
\begin{array}{l}
u_t=u_{xx},\text{\ }0<x<a,\ 0<t<T, \\
u_{x}\left( 0,t\right) =u^{-p}(0,t),\ u_{x}\left( a,t\right)
=(1-u(a,t))^{-q},\ 0<t<T, \\
u\left( x,0\right) =u_{0}\left( x\right) ,\ 0\leq x\leq a,%
\end{array}%
\right.  \label{HeatEqSBC}
\end{equation}

It shown that if $u_{0}$ satisfies $u_{xx}(x,0)\leq 0$ then $\underset{t\rightarrow T^{-}}{\lim }u(0,t)\rightarrow 0$ and $u_{t}(0,t)$ blows up in finite time and the quenching location is at $x=0$.  Likewise, it was shown that if $
u_{0}$ satisfies $u_{xx}(x,0)\geq 0$ then quenching will occur at $x=a$.

In this paper, new estimates are derived for quenching rates.  In addition, we provide necessary conditions that guarantee quenching at one of the boundary locations for a more general $\phi(u)$ and $r\geq 2$.

In the following, the initial condition may satisfy either of the two conditions:
\begin{eqnarray}
u_{xx}(x,0) &\geq &0,0<x<a,\  \mbox{or}  \label{uxxpos} \\
u_{xx}(x,0) &\leq &0,0<x<a,  \label{uxxneg}
\end{eqnarray}
and the following condition:
\begin{equation}
u_{x}(x,0)\geq 0,0<x<a.  \label{uxpos}
\end{equation}
These assumptions will be shown to guarantee that quenching will occur in finite time.

Quenching problems have a long history in applied mathematics
literature, dating back to pioneering work of Kawarada \cite{Kawarada_1975},
which examines the one-dimensional heat equation with a nonlinear source
term with Dirichlet boundary conditions. The Kawarada equations and
extensions have been a subject of interest of both numerical \cite{Beauregard_2015, Sheng_2003, Sheng_2012} and theoretical \cite{Chan_1994,
Chan_2011,Chan_2001, Selcuk_2014, Ozalp_2015, Zhi_2007}. In many situations, the location that quenching occurs may be difficult to obtain. Here, the situation of a
singular boundary condition enables theoretical predications to happen since
the quenching location is known based on simple requirements on the initial
conditions.

Chan and Yuen \cite{Chan_2001} investigated a slightly different left boundary condition:
\begin{equation*}
\begin{array}{l}
u_{t}=u_{xx},\  \text{in}\  \Omega , \\
u_{x}\left( 0,t\right) =(1-u(0,t))^{-p},\ u_{x}\left( a,t\right)
=(1-u(a,t))^{-q},\ 0<t<T, \\
u\left( x,0\right) =u_{0}\left( x\right) ,\ 0\leq u_{0}\left( x\right) <1,\
\text{in }\overset{\_}{D},%
\end{array}%
\end{equation*}
where $a,p,q>0,T\leq \infty ,D=(0,a),\Omega =D\times (0,T)$. In \cite{Chan_2001}, they showed that $x=a$ is the unique quenching point in finite
time if $u_{0}$ is a lower solution, and $u_{t}$\ blows up at quenching
time. In \cite{Selcuk_2014}, Selcuk and Ozalp considered the problem
\begin{equation*}
\begin{array}{l}
u_{t}=u_{xx}+(1-u)^{-p},\ 0<x<1,\ 0<t<T, \\
u_{x}\left( 0,t\right) =0,\ u_{x}\left( 1,t\right) =-u^{-q}(1,t),\ 0<t<T, \\
u\left( x,0\right) =u_{0}\left( x\right) ,\ 0<u_{0}\left( x\right) <1,\
0\leq x\leq 1.%
\end{array}%
\end{equation*}
They showed that $x=0\ $is the quenching point in finite time, $\lim
\nolimits_{t\rightarrow T^{-}}u(0,t)\rightarrow 1$, if $u\left( x,0\right) \
$satisfies $u_{xx}(x,0)+\left( 1-u(x,0)\right) ^{-p}\geq 0$ and $%
u_{x}(x,0)\leq 0$.$\ $Further they showed that$\ u_{t}\ $blows up at
quenching time. Furthermore, they obtained a quenching rate and a lower
bound for the quenching time. In \cite{Li_2012}, Li and et.al. considered
the quenching problem for non-Newtonian filtration equation with a singular
boundary condition
\begin{equation}
\left \{
\begin{array}{l}
(\psi (u))_{t}=(\left \vert u_{x}\right \vert ^{r-2}u_{x})_{x},\text{\ }%
0<x<1,\ 0<t<T, \\
u_{x}\left( 0,t\right) =0,\ u_{x}\left( 1,t\right) =-g(u(1,t)),\ 0<t<T, \\
u\left( x,0\right) =u_{0}\left( x\right) ,\ 0\leq x\leq 1,%
\end{array}%
\right.  \label{NN_Li}
\end{equation}
where $\psi (u)$ is a monotone increasing function with $\psi
(0)=0,p>1,g(u)>0,g^{\prime }(u)<0$ for $u>0$, and $\lim_{u\rightarrow
0^{+}}g(u)=\infty $. They showed that $x=1\ $is the only quenching point in
finite time under proper conditions, Further, they obtained a quenching rate
and gave an example of an application of their results.

In this paper, the quenching problem, Eq.~(\ref{NN_EqSBC}), exhibits two
types of singularity terms; the boundary outflux terms$\ u^{-p}\ $and $%
(1-u)^{-q}\ $as Eq.~(\ref{HeatEqSBC}). Motivated by problems (\ref{HeatEqSBC}) and (\ref{NN_Li}), we investigate the quenching behavior of Eq.~(\ref{NN_EqSBC}). Further, in such case, several questions remain open for Eq.~(%
\ref{HeatEqSBC})\ in \cite{Selcuk_2015}, in particular:

\begin{enumerate}
\item What are the quenching rates?
\item What are the estimated quenching times?
\end{enumerate}

This paper is arranged as follows. In Section 2, it will be shown that the
solution quenches in finite time\ $T$ and $\underset{t\rightarrow T^{-}}{%
\lim }u(a,t)\rightarrow \infty \ $or $\underset{t\rightarrow T^{-}}{\lim }%
u(a,t)-\infty $ blows up at quenching time at the only quenching point $x=a\
$or $x=0\ $under the conditions (\ref{uxxpos})\ or (\ref{uxxneg}),
respectively, for $r>2$. In Section 3, quenching rates are obtained of the
solution near the quenching time for $\phi (u)=u\ $and$\ r=2$. Lower bounds
to the are then given. Section 4 details the development of the finite
difference numerical approximation. Section 4 provides numerical experiments
that provide experimental validation to our theoretical results shown in
Section 3. We highlight our main results in Section 4.%

\section{Quenching for the non-Newtonian filtration equation}

Firstly, we rewrite Eq. (\ref{NN_EqSBC}) into the following form:
\begin{equation}
\left \{
\begin{array}{l}
u_{t}=B(u)(\left \vert u_{x}\right \vert ^{r-2}u_{x})_{x},\text{\ }0<x<a,\
0<t<T, \\
u_{x}\left( 0,t\right) =u^{-p}(0,t),\ u_{x}\left( a,t\right)
=(1-u(a,t))^{-q},\ 0<t<T, \\
u\left( x,0\right) =u_{0}\left( x\right) ,\ 0\leq x\leq a,%
\end{array}%
\right.   \label{NN_HeatEqSBC}
\end{equation}
where $r\geq 2$, $B(u)=1/\phi ^{\prime }(u)\ $and $\phi ^{\prime }(u)\neq 0\
$for $u>0$.

\bigskip

\begin{lemma}
~
\label{LemmaquenchforNN}
\begin{enumerate}[label=(\alph*)]
\item Assume that (\ref{uxpos}) holds. Then, $u_{x}(x,t)>0$ in $(0,a)\times (0,T_{0})$.

\item Assume that (\ref{uxxneg}) holds. Then, $u_{t}(x,t)<0$ in $(0,a)\times (0,T_{0})$.

\item Assume that (\ref{uxxpos}) holds. Then, $u_{t}(x,t)>0$ in $(0,a)\times (0,T_{0})$.
\end{enumerate}
\end{lemma}

\begin{proof}
~
\begin{enumerate}[label=(\alph*)]
\item Let $z(x,t)=u_{x}(x,t)$. Then, $z(x,t)$ satisfies
\begin{equation*}
\begin{array}{l}
z_{t}=B(u)(\left \vert z\right \vert ^{r-2}z)_{xx}+B^{\prime }(u)z(\left
\vert z\right \vert ^{r-2}z)_{x},\text{\ }0<x<a,\ 0<t<T_{0}, \\
z\left( 0,t\right) =u^{-p}(0,t),\ z\left( a,t\right) =(1-u(a,t))^{-q},\
0<t<T_{0}, \\
z\left( x,0\right) =u_{0}^{^{\prime }}(x)\text{.}%
\end{array}%
\end{equation*}
From the Maximum Principle, it follows that $z>0$ and hence $u_{x}(x,t)>0$
in $(0,a)\times (0,T_{0})$.

\item Let $w(x,t)=u_{t}(x,t)$. Then, $w(x,t)$ satisfies on $0<x<a$ and $0<t<T_0$:
\begin{equation*}
\begin{array}{l}
w_{t}=B^{\prime }(u)(\left \vert u_{x}\right \vert
^{r-2}u_{x})_{x}w+(r-1)B(u)(\left \vert u_{x}\right \vert ^{r-2}w_{x})_{x},
\end{array}%
\end{equation*}
and
\begin{equation*}
\begin{array}{l}
w_{x}\left( 0,t\right) =-pu^{-p-1}(0,t)w(0,t),\ 0<t<T_{0}, \\
w_{x}\left( a,t\right)
=q(1-u(a,t))^{-q-1}w\left( a,t\right) ,\ 0<t<T_{0}, \\
w\left( x,0\right) =B(u_{0}\left( x\right) )\left( \left \vert
u_{0}^{^{\prime }}(x)\right \vert ^{r-2}u_{0}^{^{\prime }}(x)\right)
_{x},0\leq x\leq a.
\end{array}
\end{equation*}
From the Maximum Principle, it follows that $w<0$ and hence $u_{t}(x,t)<0$ in $(0,a)\times (0,T_{0})$.

\item Similarly, $u_{0}(x)$ assumes (\ref{uxxpos}), then from the
above proof we have $u_{t}(x,t)>0$ in $(0,a)\times (0,T_{0})$. The proof is
complete.
\end{enumerate}
\end{proof}

\begin{theorem}
~
\label{quenchforNN}
\begin{enumerate}[label=(\alph*)]
 \item Assume that (\ref{uxxneg}) and (\ref{uxpos}) hold. Then, the solution $u$ of Eq.~(\ref{NN_HeatEqSBC}) quenches at time $T$. Then quenching occurs only at the boundary $x=0$ and $u_{t}(0,t)$ blows up at the quenching time.

\item Assume that (\ref{uxxpos}) and (\ref{uxpos}) hold. Then, $u_{t}(x,t)>0$ in $(0,a)\times (0,T_{0})$ and there exists a finite time $T$, such that the solution $u$ of Eq.~(\ref{NN_HeatEqSBC}) quenches at time $T$. Then quenching occurs only at the boundary $x=a$ and $u_{t}(a,t)$ blows up at the quenching time.
\end{enumerate}
\begin{proof}
~
\begin{enumerate}[label=(\alph*)]
\item Assume that (\ref{uxxneg}) holds. Then, by Lemma \ref{LemmaquenchforNN}(b), we get $u_{t}(x,t)<0$ in $(0,a)\times (0,T_{0})$. In addition, by (\ref{uxxneg}):
\begin{equation*}
\omega _{3}=-(1-u\left( a,0\right) )^{-q(r-1)}+u^{-p(r-1)}\left( 0,0\right)
>0\text{.}
\end{equation*}
We shall introduce a mass function:
$$ m_{3}\left( t\right) =\displaystyle \int_{0}^{a}\phi (u\left(
x,t\right) )dx,0<t<T.$$
Then
\begin{equation*}
m_{3}^{\prime }\left( t\right) =(1-u\left( a,t\right)
)^{-q(r-1)}-u^{-p(r-1)}\left( 0,t\right) \leq -\omega _{3},
\end{equation*}
by $u_{t}(x,t)<0\ $in $(0,a)\times (0,T_{0})$. Thus, $m_{3}\left( t\right)
\leq m_{3}(0)-\omega _{3}t$; which means that $m_{3}\left( T_{0}\right) =0\ $%
for some $T_{0},(0<T\leq T_{0})$ which means $u$ quenches in finite time.

Since $r\geq 2$, $\phi (u)\ $is an increasing function, $u_{t}(x,t)<0$ and $%
u_{x}(x,t)>0\ $in $(0,a)\times (0,T_{0})$, we get
\begin{eqnarray*}
(\phi (u))_{t}=(\left \vert u_{x}\right \vert ^{r-2}u_{x})_{x}&\rightarrow&
\phi ^{\prime }(u)u_{t}=(r-1)u_{x}^{r-2}u_{xx}\\
&\rightarrow& u_{xx}=\frac{\phi
^{\prime }(u)u_{t}}{(r-1)u_{x}^{r-2}}<0.
\end{eqnarray*}
Namely, $u_{x}\ $is a decreasing function and since $u_{x}(a,t)=(1-u(a,t))^{-q}>1$, $u_{x}(x,t)>1\ $in $(0,a)\times (0,T)$. Let $\eta \in (0,a)$. Integrating this with respect to $x$ from $0$ to $\eta $,
we have
\begin{equation*}
u(\eta ,t)>u(0,t)+\eta >0\text{.}
\end{equation*}
So $u\ $does not quench in $(0,a]$.

Suppose that $u_{t}\ $is bounded in $[0,a)\times \lbrack 0,T)$. Then there
is a positive constant $M$, $u_{t}>-M$. Therefore,
\begin{equation*}
B(u)(\left \vert u_{x}\right \vert ^{r-2}u_{x})_{x}>-M.
\end{equation*}
Because of $\phi ^{\prime \prime }(s)<0$, $\phi ^{\prime }(s)\ $is
not increasing. So, there are $\sigma \ $and $\tau ,\ $which make $0<\tau
\leq v<1\ $in $[0,\sigma ]\times \lbrack 0,T)$, thus, $B(u)=\frac{1}{\phi
^{\prime }(u)}\geq B(\tau )$. Thus,
\begin{eqnarray*}
(\left \vert u_{x}\right \vert ^{r-2}u_{x})_{x} &>&\frac{-M}{B(u)}\geq \frac{%
-M}{B(\tau )}, \\
(u_{x}^{r-1})_{x} &>&\frac{-M}{B(\tau )},
\end{eqnarray*}
from $u_{x}(x,t)>0\ $in $(0,a)\times (0,T_{0})$. Integrating this with
respect to $x$ from $0$ to $a$, we have
\begin{equation*}
(1-u(a,t))^{-(r-1)q}-u^{-(r-1)p}(0,t)>\frac{-Ma}{B(\tau )}\text{.}
\end{equation*}
As $t\rightarrow T^{-}$, the left-hand side tends to negative infinity,
while the right-hand side is finite. This contradiction shows that $u_{t}$
blows up at the quenching time for $x=0$.

\item Similarly, assume that (\ref{uxxpos}) and (\ref{uxpos}) hold. From (a), we have quenching occurs only at the boundary $x=a$ and $u_{t}$ blows up at the quenching time for $x=a$. The proof is therefore complete.
\end{enumerate}
\end{proof}
\end{theorem}

\section{Quenching rates of the heat equation}

In this section, we investigate the case where $\phi(u)=u$ and $r=2$ and determine quenching rates and lower bounds to the quenching time under certain conditions on the initial condition in Eq.~(\ref{HeatEqSBC}). In the following we may either assumption on the spatial derivative of the initial condition:
\begin{eqnarray}
u_{x}(x,0)\geq \frac{x}{a}(1-u(x,0))^{-q}, &&~~0<x<a,~~\mbox{or}
\label{uxcond1} \\
u_{x}(x,0)\geq \frac{(a-x)}{a}u^{-p}(x,0), &&~~0<x<a.  \label{uxcond2}
\end{eqnarray}

\begin{theorem}
\label{quenchestrightwall} If $u_{0}(x)$ satisfies condition (\ref{uxxpos}),
that is, the initial condition is not concave down, then there exists a
positive constant $C_{1}$ such that
\begin{equation*}
u(a,t)\leq 1-C_{1}(T-t)^{1/(2q+2)},
\end{equation*}
for\ $t$ sufficiently close to the quenching time $T$.
\end{theorem}

\begin{proof}
Define
\begin{equation*}
M(x,t) = u_{t}-\delta q(1-u)^{-q-1}u_{x},
\end{equation*}
in $[0,a]\times \lbrack \tau ,T)$ where $\tau \in (0,T)\ $and $\delta $\ is
a positive constant to be specified later. It was shown in \cite{Selcuk_2015}%
, that since $u_{t}>0$ and $u_{x}>0$ in $\left( 0,a\right) \times (0,T)$,
then $M(x,t)$ satisfies
\begin{eqnarray*}
M_{t}-M_{xx}=\delta q(q+1)(q+2)(1-u)^{-q-3}u_{x}^{3}+2\delta
q(q+1)(1-u)^{-q-2}u_{x}u_{t}>0,
\end{eqnarray*}
for $(x,t)\in (0,a)\times (\tau ,T).$ Furthermore, if $\delta$ is small
enough then $M(x,\tau )\geq 0$ for $x\in [0,a]$, and $M(0,t)>0,M(a,t)>0$ for
$t\in \lbrack \tau ,T)$.

Therefore, by the maximum principle, we obtain that $M(x,t)\geq 0$ for $%
(x,t)\in \lbrack 0,a]\times \lbrack \tau ,T)$. This means that
\begin{equation*}
u_{t}(x,t)\geq \delta q(1-u)^{-q-1}u_{x}(x,t), ~~(x,t)\in \lbrack 0,a]\times
\lbrack \tau ,T)
\end{equation*}
Evaluating at $x=a$ yields,
\begin{equation*}
u_{t}(a,t)\geq \delta q(1-u(a,t))^{-2q-1}.
\end{equation*}%
Integrating over $t$ from $t$ to $T$ gives,
\begin{equation*}
u(a,t)\leq 1-C_{1}(T-t)^{1/(2q+2)},
\end{equation*}
where $C_{1}=\left( 2\delta q(q+1)\right) ^{1/(2q+2)}$.
\end{proof}

If we provide the additional condition on the spatial derivative of the
initial condition then we can obtain a lower bound to the value at the right
hand wall. This is encapsulated in the following theorem.

\begin{theorem}
\label{lowerbound} If $u_{0}(x)$ satisfies conditions (\ref{uxxpos}) and (%
\ref{uxcond1}) then there exist positive constant $C_{2}$ such that
\begin{equation*}
u(a,t)\geq 1-C_{2}(T-t)^{1/(2q+2)},
\end{equation*}
for $t$ sufficiently close to the quenching time $T$.
\end{theorem}

\begin{proof}
Define%
\begin{equation*}
J(x,t)=u_{x}-\frac{x}{a}(1-u)^{-q}, ~~(x,t)\in [0,a]\times \lbrack 0,T).
\end{equation*}%
Then, $J(x,t)$ satisfies
\begin{equation*}
J_{t}-J_{xx}=\frac{1}{a}\left(
2q(1-u)^{-q-1}u_{x}+xq(q+1)(1-u)^{-q-2}u_{x}^{2}\right).
\end{equation*}
$J(x,t)$ cannot attain a negative interior minimum since $u_x(x,t)>0$. On
the other hand, by our condition (\ref{uxcond1}) we have $J(x,0)\geq 0$ and
\begin{equation*}
J(0,t)=u^{-p}(0,t)>0,~~J(a,t)=0,
\end{equation*}
for\ $a\leq 1$ and $t\in (0,T)$. By the maximum principle, we obtain that $%
J(x,t)\geq 0$ for $(x,t)\in \lbrack 0,1]\times \lbrack 0,T)$. Therefore,
\begin{equation*}
J_{x}(a,t)=\underset{h\rightarrow 0^{+}}{\lim }\frac{J(a,t)-J(a-h,t)}{h}=%
\underset{h\rightarrow 0^{+}}{\lim }\frac{-J(a-h,t)}{h}\leq 0\text{.}
\end{equation*}
Subsequently,
\begin{eqnarray*}
J_{x}(a,t) &=&u_{xx}(a,t)-\frac{1}{a}(1-u(a,t))^{-q}-q(1-u(a,t))^{-2q-1} \\
&=&u_{t}(a,t)-\frac{1}{a}(1-u(a,t))^{-q}-q(1-u(a,t))^{-2q-1}\leq 0
\end{eqnarray*}
and
\begin{equation*}
u_{t}(a,t)\leq \frac{(qa+1)}{a}(1-u(a,t))^{-2q-1}.
\end{equation*}
Integrating over $t$ from $t$ to $T$ yields
\begin{equation*}
u(a,t)\geq 1-C_{2}(T-t)^{1/(2q+2)},
\end{equation*}
where $C_{2}=\left[ \frac{(qa+1)(2q+2)}{a}\right] ^{1/(2q+2)}$.
\end{proof}

\begin{cor}
\label{cor1} The results of the Theorems (\ref{quenchestrightwall}) and (\ref%
{lowerbound}) suggest as the quenching time is approached that the quenching
rate of the solution can be estimated as
\begin{equation*}
u(a,t)\thicksim 1-\left( T-t\right) ^{\dfrac{1}{2(q+1)}}\text{.}
\end{equation*}
Equivalently,
\begin{equation*}
\dfrac{\ln(1-u(a,t))}{\ln(T-t)} \thicksim \frac{1}{2(q+1)}
\end{equation*}
In addition, a lower bound for the quenching time can be calculated. From
Theorem (\ref{lowerbound}), we have
\begin{equation*}
T_{q}=\frac{a(1-u_{0}(a))^{2q+2}}{2(qa+1)(q+1)} \leq T.
\end{equation*}
\end{cor}

In the following, we assume the initial condition satisfies condition (\ref%
{uxxneg}). This condition guarantees quenching will occur at the left
boundary, $x=0$. Hence, we seek quenching estimates to the quenching rate of
the solution.

\begin{theorem}
\label{quenchingestimateleftwall} If $u_{0}(x)$ satisfies condition $(\ref%
{uxxneg})$, that is, the initial condition is not concave up, then there
exists a positive constant $C_{3}$ such that
\begin{equation*}
u(0,t)\geq C_{3}(T-t)^{1/(2p+2)},
\end{equation*}
for $t$ sufficiently close to the quenching time $T$.
\end{theorem}

\begin{proof}
Define
\begin{equation*}
M(x,t) = u_{t}+\delta pu^{-p-1}u_{x}, ~~ (x,t) \in [0,a]\times \lbrack \tau
,T)
\end{equation*}
where $\tau \in (0,T)$ and $\delta$ is a positive constant to be specified
later. It was shown in \cite{Selcuk_2015} that since $u_{t}<0$ and $u_{x}>0$
in $\left( 0,a\right) \times (0,T)$ then $M(x,t)$ satisfies
\begin{equation*}
M_{t}-M_{xx}=-\delta p(p+1)(p+2)u^{-p-3}u_{x}^{3}+2\delta
p(p+1)u^{-p-2}u_{x}u_{t}<0,
\end{equation*}
for $(x,t)\in (0,a)\times (\tau ,T).$ Furthermore, if $\delta$ is small
enough, then $M(x,\tau )\leq 0$ for $x\in [0,a]$ and $M(0,t)<0,~M(a,t)<0$
for $t\in \lbrack \tau ,T)$. Therefore, by the maximum principle, we obtain
that $M(x,t)\leq 0$ for $(x,t)\in \lbrack 0,a]\times \lbrack \tau ,T).$
Subsequently, $u_{t}(x,t)\leq -\delta pu^{-p-1}u_{x}(x,t)$ for $(x,t)\in
\lbrack 0,a]\times \lbrack \tau ,T)$. This means, at $x=0$ we have:
\begin{equation*}
u_{t}(0,t)\leq -\delta pu^{-2p-1}(0,t).
\end{equation*}
Integrating over $t$ from $t$ to $T$ yields,
\begin{equation*}
u(0,t)\geq C_{3}(T-t)^{1/(2p+2)},
\end{equation*}
where $C_{3}=\left( 2\delta p(p+1)\right) ^{1/(2p+2)}$.
\end{proof}

\begin{theorem}
\label{lowerbound2} If $u_{0}(x)$ satisfies both (\ref{uxxpos}) and (\ref%
{uxcond2}) then there exist positive constant $C_{4}$ such that
\begin{equation*}
u(0,t)\leq C_{4}(T-t)^{1/(2p+2)},
\end{equation*}
for $t$ sufficiently close to the quenching time $T$.
\end{theorem}

\begin{proof}
Define
\begin{equation*}
J(x,t)=u_{x}-\frac{(a-x)}{a}u^{-p},~~(x,t) \in [0,a]\times \lbrack 0,T).
\end{equation*}
Then, $J(x,t)$ satisfies
\begin{equation*}
J_{t}-J_{xx}=\frac{1}{a}\left(
2pu^{-p-1}u_{x}+(a-x)p(p+1)(1-u)^{-p-2}u_{x}^{2}\right).
\end{equation*}
Since $u_{x}>0$, then $J(x,t)$ cannot attain a negative interior minimum. On
the other hand, by the assumed condition (\ref{uxcond2}), then $J(x,0)\geq 0$
and
\begin{equation*}
J(0,t)=0,J(a,t)=(1-u(a,t))^{-q}>0,
\end{equation*}
for $t\in (0,T)$. Therefore, by the maximum principle, we obtain that $%
J(x,t)\geq 0$ for $(x,t)\in \lbrack 0,1]\times \lbrack 0,T)$. As a result,
\begin{equation*}
J_{x}(0,t)=\underset{h\rightarrow 0^{+}}{\lim }\frac{J(h,t)-J(0,t)}{h}=%
\underset{h\rightarrow 0^{+}}{\lim }\frac{J(h,t)}{h}\geq 0.
\end{equation*}
This yields
\begin{eqnarray*}
J_{x}(0,t) &=&u_{xx}(0,t)+\frac{1}{a}u^{-p}(0,t)+pu^{-2p-1}(0,t) \\
&=&u_{t}(0,t)+\frac{1}{a}u^{-p}(0,t)+pu^{-2p-1}(0,t)\geq 0
\end{eqnarray*}
and
\begin{equation*}
u_{t}(0,t)\geq -\frac{(pa+1)}{a}u^{-2p-1}(0,t).
\end{equation*}
Integrating from $t$ from $t$ to $T$ gives
\begin{equation*}
u(0,t)\leq C_{4}(T-t)^{1/(2p+2)},
\end{equation*}
where $C_{4}=\left[ \frac{(pa+1)(2p+2)}{a}\right] ^{1/(2p+2)}$.
\end{proof}

\begin{cor}
\label{cor2} The results of the Theorems (\ref{quenchingestimateleftwall})
and (\ref{lowerbound2}) suggest as the quenching time is approached that the
quenching rate of the solution is estimated as
\begin{equation*}
u(0,t)\thicksim \left( T-t\right) ^{1/(2p+2)}
\end{equation*}
Equivalently,
\begin{equation*}
\dfrac{\ln(u(0,t))}{\ln(T-t)} \thicksim \frac{1}{2(p+1)}
\end{equation*}
In addition, a lower bound for the quenching time is established from
Theorem (\ref{lowerbound2}), namely,
\begin{equation*}
T_{p}=\frac{au_{0}(0))^{2p+2}}{2(pa+1)(p+1)} \leq T.
\end{equation*}
for quenching time\ $T$.
\end{cor}

\bigskip

\subsection{Initial Conditions Examples}

It is clear, that the estimates for the quenching rates and times rely
heavily on properties of the initial condition. Here, we provide initial
functions that satisfy the boundary conditions while simultaneously
satisfying either conditions (\ref{uxxpos}) and (\ref{uxcond1}) or (\ref%
{uxxneg}) and (\ref{uxcond2}).

Consider the initial condition,
\begin{equation}  \label{ICExample1}
u_{0}(x)=\frac{1}{4}+4x+4x^{2},~~ 0\leq x\leq a.
\end{equation}
where $a=1/8$. Let $p=1$ and $q=\log_{16/3}(5)$. Since the initial condition
is concave up throughout its entire domain then clearly condition (\ref%
{uxxpos}) is satisfied. In addition, a straightforward calculation shows
that the left boundary condition is satisfied, namely,
\begin{equation*}
u_0^{\prime }(0) = 4 = \dfrac{1}{u_0(0)^p}
\end{equation*}
At the right boundary we have $u_0^{\prime }\left(\frac18\right) = 5$ and
\begin{eqnarray*}
\dfrac{1}{\left( 1 - u_0\left( \frac18 \right)\right)^q} = \left( \frac{16}{3%
} \right)^q = 5
\end{eqnarray*}
In Fig.~\ref{ICExampleBound}(a) it is seen that the condition (\ref{uxcond1}%
) is satisfied.

\begin{center}
\begin{figure}[ht]
\includegraphics[scale=.45]{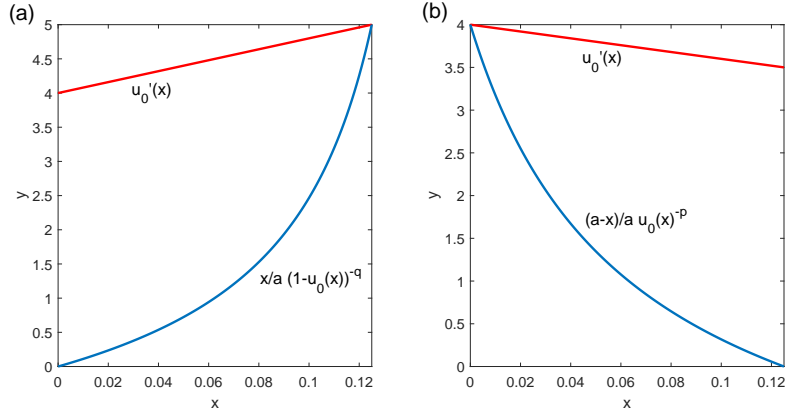} 
\caption{(a) A graph of $u_0^{\prime }(x)$ (RED) and $\frac{x}{a}%
(1-u_0(x))^q $ (BLUE) for $u_0(x) = \frac14 - 4x - 4x^2$. It is clear that $%
u_0^{\prime }(x) \geq \frac{x}{a}(1-u_0(x))^{-q}$ is satisfied throughout
the domain $0 \leq x \leq 1/8.$ (b) A graph of $u_0^{\prime }(x)$ (RED) and $%
\frac{a-x}{a}(u_0(x))^{-p}$ (BLUE) for $u_0(x) = \frac14 + 4x - 2x^2$. It is
clear that $u_0^{\prime }(x) \geq \frac{a-x}{a}(u_0(x))^{-p}$ is satisfied
throughout the domain $0 \leq x \leq 1/8.$}
\label{ICExampleBound}
\end{figure}
\end{center}

In light of the initial condition (\ref{ICExample1}) then, by Corollary (\ref%
{cor1}) we have a lower bound to quenching time. Namely:
\begin{eqnarray*}
T_{q}&=&\frac{ \left( 3/16\right)^{2q+2} }{16\left( \frac18 q + 1
\right)(q+1)} \approx 4.0002\times 10^{-5}.
\end{eqnarray*}

Similarly, if the initial condition is
\begin{equation}
u_{0}(x)=\frac{1}{4}+4x-2x^{2},~~0\leq x\leq a.  \label{ICExample2}
\end{equation}%
where $a=1/8$. Let $p=1$ and $q=\log _{32/9}\left( \frac{7}{2}\right) $.
Since the initial condition is concave down throughout its entire domain
then clearly condition (\ref{uxxneg}) is satisfied. It is clear that the
left boundary condition is satisfied. At the right boundary we have $%
u_{0}^{\prime }\left( \frac{1}{8}\right) =\left( 1-u_{0}\left( \frac{1}{8}%
\right) \right) ^{-q}=\frac{7}{2}.$ In Fig.~\ref{ICExampleBound}(b), we see
that condition (\ref{uxcond2}) is satisfied. Furthermore, by Corollary (\ref%
{cor2}) we have a lower bound to quenching time. Namely:
\begin{equation*}
T_{p}=\dfrac{1}{9216}\approx 1.0851\times 10^{-5}.
\end{equation*}

\section{Numerical Approximation and Experiments}

Let $x_{j}=jh$ for $j=0,\ldots ,N+1$ and $h=a/(N+1)$. Let $%
t_{k}=t_{k-1}+\tau _{k-1}$, where $\tau _{k-1}$ is the temporal step. Let $%
u_{j}(t)$ be the approximation to $u(x_{j},t)$. Define the vector $\vec{u}%
(t)=(u_{0}(t),u_{1}(t),\ldots ,u_{N}(t),u_{N+1}(t))^{\top}$, where $\vec{u}%
(0)$ is created from evaluating the initial condition at the grid points.
Central difference approximations are utilized at each grid point to create
the semidiscretized equations approximating Eq.~(\ref{HeatEqSBC}), namely,
\begin{equation}
h^{2}\dot{\vec{u}}(t)=\vec{F}(\vec{u}(t)),
\end{equation}%
where $\vec{F}=(F_{0},\ldots ,F_{N+1})$ with components defined as
\begin{equation}
F_{k}=\left \{
\begin{array}{lcl}
2u_{1}+\dfrac{2h}{(u_{0})^{p}}-2u_{0} &  & k=0 \\
u_{k-1}-2u_{k}+u_{k+1} &  & k=1,2,\ldots ,N \\
2u_{N}+\dfrac{2h}{(1-u_{N+1})^{q}}-2u_{N+1} & \hspace{5em} & k=N+1%
\end{array}%
\right.
\end{equation}%
Define $\vec{v}_{m}$ as the approximation to $\vec{u}(t)$ at time $t=t_{m}$.
Then, the solution is advanced through a second order accurate
Crank-Nicolson scheme \cite{Strikwerda}:
\begin{equation}
\vec{v}_{m+1}=\vec{v}_{m}+\mu _{m}(\vec{F}(\vec{v}_{m+1})+\vec{F}(\vec{v}%
_{m})),
\end{equation}%
where $\mu _{m}=\tau _{m}/(2h^{2})$. The scheme is overall second order
accurate, however, due to the singular boundary conditions the equations are
\textit{stiff} and it is known that unless $\tau _{k}$ is sufficiently then
the method may manifest a reduction in the order of temporal convergence
\cite{Hundsdorfer_1992}. With this in mind, we expect the method to overall
first order accurate modest temporal steps. It is common to approximate $%
\vec{v}_{m+1}$ in the right hand side by a first order Euler approximation, $%
\vec{v}_{m+1}\approx \vec{v}_{m}+2\mu _{m}\vec{F}(\vec{v}_{m})$. This
maintains the overall accuracy of the scheme will creating a semi-explicit
scheme for efficiency in computations \cite{Beauregard_2013}. The spatial
grid is fixed throughout the computation, however, adaptation may occur in
the temporal step. Temporal adaption for quenching problems is critical to
ensure accuracy in the quenching time. An arc-length monitoring function for
$\dot{\vec{u}}$ is used to adapt the temporal step. Define
\begin{equation*}
m_{i}\left( \frac{\partial u_{i}}{\partial t},t\right) =\sqrt{1+\left( \frac{%
\partial ^{2}u_{i}}{\partial t^{2}}\right) ^{2}},~~(x,t)\in \lbrack
0,a]\times (0,T]
\end{equation*}%
for $i=0,\ldots ,N+1$. The monitoring functions, $m_{i}$, monitor the
arc-length of the characteristic at node $x_{i}$. Subsequently, as quenching
is approached the temporal derivative grows beyond exponentially fast,
therefore the arc-length will grow \cite{Beauregard_2013_AAMM}. Therefore,
we choose the temporal step such that the maximal arc-length between
successive approximations at $[t_{k-2},t_{k-1}]$ and $[t_{k-1},t_{k}]$ are
equivalent. Pragmatically, this leads to the equation for the temporal step:
\begin{equation*}
\tau _{k}^{2}=\tau _{k-1}^{2}+\min_{i}\left \{ \left[ \left( \frac{\partial
u_{i}}{\partial t}\right) _{k-1}-\left( \frac{\partial u_{i}}{\partial t}%
\right) _{k-2}\right] ^{2}-\left[ \left( \frac{\partial u_{i}}{\partial t}%
\right) _{k}-\left( \frac{\partial u_{i}}{\partial t}\right) _{k-1}\right]
^{2}\right \} ,
\end{equation*}%
for $~k=2,\ldots ,$ and given the initial times steps of $\tau _{0}$ and $%
\tau _{1}$.

In the following experiments, we look to verify the second order convergence
rate of the numerical routine. Assume that $t\ll T$. Let $\vec{v}_{\tau}$ be
the approximation to $\vec{u}(\tau)$ for a fixed temporal step $\tau.$ Then,
the maximum absolute difference between the numerical solution and $\vec{u}$
at time is $\max |\vec{v}_{\tau} - \vec{u}| \approx C \tau^p$, where $C$ is
some positive constant and $p$ is the order of accuracy of the temporal
scheme. Consider creating a new approximation with a temporal step $\tau/2$,
then at each grid point,
\begin{eqnarray*}
|(\vec{v}_{\tau/2} - \vec{u})_i| &\approx & C \left(\dfrac{h}{2}\right)^p =
\frac{C h^p}{2^p} \\
&\approx& \frac{|(\vec{v}_{\tau} - \vec{u})_i|}{2^p}
\end{eqnarray*}
for $i=0,\ldots, N+1$. Rearranging, yields an expression to estimate the
order of accuracy,
\begin{equation*}
p \approx \frac{1}{\ln(2)} \ln \left(\dfrac{ |(\vec{v}_{\tau} - \vec{u})_i|
}{|(\vec{v}_{\tau/2} - \vec{u})_i|}\right)
\end{equation*}
This generates an approximate convergence rate at each grid point $x_i$. In
the majority of applications $\vec{u}$ is unknown. Hence, a numerical
solution with a relatively fine temporal step is used to estimate the rate
of the underlying cauchy sequence \cite{Padgett_2017}.


Consider the initial condition Eq.~\ref{ICExample1}, where $a=1/8$, $p=1$,
and $q=\log_{16/3}(5).$ We choose $\tau = 10^{-4}$ and $h=.01$. In such
case, we estimate the convergence rate of $1.013$. Therefore, a reduction in
the temporal order of convergence is manifested. To estimate the quenching
time and rates, we run the simulation with $h=.001$ and $\tau_0 = \tau_1 =
10^{-6}$. We adapt the temporal step but require $\tau_k \geq 10^{-9}$. The
quenching time is numerically determined to be approximately $T\approx
1.9037\times 10^{-3}$ which is greater than our estimated lower bound of $%
4\times 10^{-5}.$ A loglog plot of $1-u(1/8,t)$ versus $T-t$ is shown in
Fig.~\ref{QuenchingRate}(a). A least squares approximation suggests a slope
of approximately $0.253286153170844$. The theoretical estimate was
predicated to be $0.255$.

\begin{center}
\begin{figure}[ht]
\includegraphics[scale=.65]{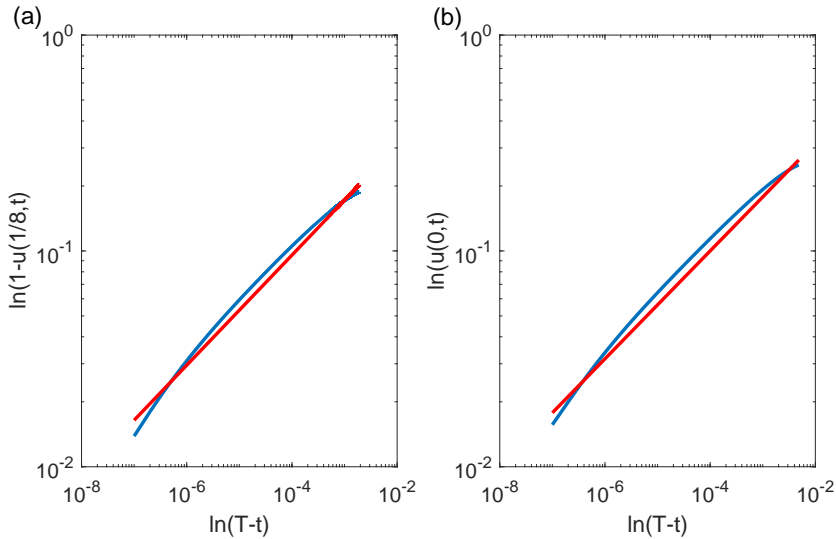} 
\caption{Loglog plots of the numerical observed (a) $1-u(a,t)$ and (b) $%
u(0,t)$ versus $T-t$. The red curves in each subplot provide a loglog of (a)
$(T-t)^{1/(2(q+1))}$ and (b) $(T-t)^{1/(2(p+1))}$ .}
\label{QuenchingRate}
\end{figure}
\end{center}

Next, consider the initial condition Eq.~\ref{ICExample2}, where $a=1/8$, $%
p=1$, and $q=\log_{32/9}(7/2).$ Again, we run the simulation with $h=.001$
and $\tau_0 = \tau_1 = 10^{-6}$. We adapt the temporal step but require $%
\tau_k \geq 10^{-9}$. The quenching time is numerically determined to be
approximately $T\approx \times 10^{-3}$ which is greater than our estimated
lower bound of $1.0851\times 10^{-5}.$ A loglog plot of $u(0,t)$ versus $T-t$
is shown in Fig.~\ref{QuenchingRate}(b). A least squares approximation
suggests a slope of approximately $0.244301262418202$. The theoretical
estimate was predicated to be $0.25$.


\section{Conclusions}

In this paper, a quenching problem with nonlinear boundary conditions are
investigated. Certain conditions on the positivity, concavity, and the first
derivative of the initial condition lead to theoretical lower bound to the
quenching time, in addition to asymptotic estimates to the quenching rate.
Numerical experiments provided additional validation of the pragmatic
application of the theoretical analysis. We found that the experimental
quenching time, $T$, was later than our predicted lower bound. Further, the
experiments suggested quenching rates that were within $1\%$ of the
predicted asymptotic quenching rates.


\end{document}